\documentclass[letterpaper, 10 pt, conference]{ieeeconf}  
\IEEEoverridecommandlockouts 
\usepackage{amsmath,amssymb,amsfonts}
\usepackage{graphicx}
\usepackage{textcomp}
\usepackage{xcolor}
\usepackage{amsthm}
\newtheorem{theorem}{Theorem}

\newtheorem{assumption}{Assumption}
\newtheorem{proposition}{Proposition}
\newtheorem{remark}{Remark}

\usepackage{algorithmic}
\usepackage{algorithm}

\usepackage{cite}
\usepackage{mathtools}
\usepackage{enumerate}
\usepackage{subcaption}

\title{\LARGE \bf
Distributed Combined Space Partitioning and Network Flow Optimization: an Optimal Transport Approach
}

\author{Théo Laurentin$^{1,2}$, Patrick Coirault$^{2}$, Emmanuel Moulay$^{3}$, Antoine Lesage-Landry$^{1}$, and Jerome Le Ny$^{1}$  
\thanks{This work was supported by NSERC under grants ALLRP 586247-23 and RGPIN-5287-2018, 
and by the R\'egion Nouvelle-Aquitaine under grant AAPR-2023-2022-23896910.}
\thanks{$^{1}$Department of Electrical Engineering, Polytechnique Montreal and GERAD, Montreal, QC H3T\,1J4, Canada. Emails: {\tt\small \{theo.laurentin, antoine.lesage-landry, jerome.le-ny\}@polymtl.ca}}%
\thanks{$^{2}$LIAS (UR 20299), Universit\'e de Poitiers, 
2 rue Pierre Brousse, 86073 Poitiers Cedex 9, France. Email: {\tt\small patrick.coirault@univpoitiers.fr}}%
\thanks{$^{3}$XLIM (UMR CNRS 7252), Universit\'e de Poitiers, 
11 bd Marie et Pierre Curie, 86962 Futuroscope Chasseneuil Cedex, France. 
Email: {\tt\small emmanuel.moulay@univpoitiers.fr.}}%
}

\begin{document}
\maketitle

\begin{abstract}
This paper studies a combined space partitioning
and network flow optimization problem, with applications to
large-scale electric power, transportation, or communication systems.
In dense wireless networks for instance, one may want to simultaneously optimize the assignment of many spatially distributed
users to base stations and route the resulting communication
traffic through the backbone network. We formulate the overall
problem by coupling a semi-discrete optimal transport (SDOT)
problem, capturing the space partitioning component, with
a minimum-cost flow problem on a discrete network. This
formulation jointly optimizes the assignment of a continuous
demand distribution to certain endpoint network nodes and
the routing of flows over the network to serve the demand,
under capacity constraints. As for SDOT problems, we establish
that the formulation of our problem admits a tight relaxation
taking the form of an infinite-dimensional linear program,
derive its finite-dimensional dual, and prove that strong duality holds.
We leverage this theory to design a distributed dual 
(super)gradient ascent algorithm solving the problem,
where nodes in the graph perform computations based solely on
locally available information. Simulation results illustrate the
algorithm’s performance and its applicability to an electric
power distribution network reconfiguration problem. 

\end{abstract}

\section{Introduction}

Optimizing large-scale networks requires solving two fundamental problems: allocating end resources to serve the demand for a product or service, and efficiently routing flows through the network to ultimately meet that demand. For instance, in electric power systems, consumers  are assigned to a substation of the distribution system and the necessary power flows are routed through the transmission network to feed the substations. Similarly, in cellular communication networks, users must be matched to base stations and the resulting traffic must be routed through the backbone network. Although the assignment and routing problems are often treated separately for computational efficiency, considering both jointly can improve performance \cite{ReviewTSODSO,jointUA_CellularNetwork}.

This work addresses such a combined assignment and flow optimization problem, from an optimal transport (OT) point of view~\cite{peyré2019computationaloptimaltransport,Villani:2003book:topicsOT}. Specifically, we couple a semi-discrete optimal transport (SDOT) problem ~\cite[Chapter 5]{peyré2019computationaloptimaltransport} with a minimum-cost flow (MCF) problem~\cite{bertsekas1998network}, and develop a distributed algorithm to solve the combined problem. In the standard SDOT problem, one seeks an optimal transport map from an arbitrary source measure, representing for instance a demand distribution over a geographic area, to a discrete target measure, e.g., a set of endpoint nodes capable of serving this demand. This map in effect partitions the space supporting the demand distribution into different cells, such that the total demand in each cell equals the capacity of the associated endpoint node. Meanwhile, the MCF problem allows us to determine how to route flows efficiently through a network connecting the endpoints to supply nodes, subject to edge capacity constraints.

Problems related to the one formulated in this paper have been considered in several application domains. For example, in electrical distribution networks comprising substations and consumers connected via lines equipped with tie and sectionalizing switches, the radial reconfiguration problem~\cite{ReconfigurationSurvey}, further discussed in Section~\ref{section: electrical network ex}, aims to selectively open and close switches to form a radial topology, with substations acting both as root nodes of the distribution network and as interfaces to the transmission network where generation is available. This reconfiguration process results in a spatial partitioning that guarantees demand satisfaction while minimizing power losses. Conventional approaches to address it rely on solving computationally difficult mixed-integer programs or use faster heuristics that yield suboptimal solutions \cite{ReconfigurationSurvey}. Recent studies have explored clustering~\cite{ClusteringPowerGrids, ReconfBranchExchangeCluster} and partitioning techniques~\cite{PartitioningRadialReconf} to mitigate computational challenges, but retain a purely discrete network modelling approach. In contrast, the SDOT framework adopted here can provide good heuristic solutions for asymptotically large distribution networks by considering the limit of a continuous distribution of customers. 

In communication networks, related problems combining user association and resource allocation have been formulated as difficult combinatorial problems~\cite{jointUA_CellularNetwork,joint_routing}. Alternatively, SDOT has been used for space partitioning, 
e.g., for C-RAN device association~\cite{OT_5G} and 
communication with UAVs\cite{OT_UAVs}. However, these papers did not integrate network flow optimization.
Another related problem 
is branched optimal transport (BOT) \cite{branchedOT}, 
which uses subadditive costs in an OT formulation to encourage mass transport along paths,
effectively yielding branched networks. 
However, in BOT, the backbone network needs to be designed, whereas it is fixed and given here.
Moreover, exact methods are limited to small instances and heuristics are generally required \cite{branchedOT}.

To ensure scalability to large-scale networks, we propose a \emph{distributed} optimization method 
to solve our problem, where an agent placed at each node of the backbone network updates its 
local variables by communicating only with neighbouring agents. Distributed methods for various OT problems 
have been proposed in recent years. Decentralized alternating direction methods of multipliers (ADMM)-based methods can be used in discrete 
settings \cite{DEOT}, while \cite{DistributedOnlineOT} propose a distributed online optimization 
and control strategy 
for a continuous problem. 
For SDOT, dual deterministic and stochastic gradient ascent methods 
can be implemented in a distributed manner by computing assignment cell areas exactly or via sampling methods
\cite{SDOT_Distributed, lévyDistributedSDOT}. MCF problems have similarly been solved using distributed auction-based algorithms \cite{BertsekasDistributed} 
or 
distributed dual gradient ascent when cost functions are strictly convex \cite{bertsekas1998network}.
Under additional regularity conditions, \cite{DistributedNewtonMCF} develops a 
distributed Newton method. The approach proposed here relies on a dual gradient ascent to accommodate both the SDOT and MCF components.

The main contributions of this work are threefold. First, we propose a new framework 
combining SDOT and MCF, enabling joint space partitioning and flow optimization 
for the backbone network. Specifically, our model extends SDOT by allowing the target 
distribution to vary according to network constraints, a significant departure 
from standard formulations where this distribution is fixed. Second, leveraging optimal 
transport theory, we derive the dual problem and show that strong duality holds. 
Third, we develop a distributed dual ascent algorithm to solve the combined SDOT–MCF problem, 
suitable for large-scale networks. 

The rest of this paper is organized as follows. Section \ref{section:formulation} formulates  the problem. Section \ref{section:relaxation} introduces a Kantorovich-type relaxation \cite{Villani:2003book:topicsOT}  of the problem, derives its dual, shows that strong duality holds, and that the relaxation is tight.  Section \ref{section:algorithm} develops the distributed dual gradient ascent algorithm. Simulation results are presented in Section~\ref{section:simulations}, and concluding remarks are provided in Section~\ref{section:conclusion}.

\section{Problem formulation}
\label{section:formulation}

We now formally define our optimization problem.
We consider a space $X$ equipped with a 
non-negative measure~$\mu$, which models, for instance, a demand distribution 
for a certain product, e.g., electric power, by a population of consumers spatially distributed over $X$. 
In addition, we consider a network modelled by a directed graph $G = (\mathcal{N}, \mathcal{A})$, 
where $\mathcal{N} \subset \mathbb{N}$ is a finite set of nodes and $\mathcal{A} \subset \mathcal{N} \times \mathcal{N}$ 
is a set of arcs. Each node $i\in \mathcal{N}$ has an associated net \emph{supply} 
value $s_i \in \mathbb R$ (with $s_i < 0$ allowed, indicating a demand node).
A subset of the nodes, denoted by $\mathcal{S}\subset \mathcal{N}$, corresponds to \emph{endpoint} stations effectively serving the consumers, e.g., distribution substations in power systems. Serving a unit of demand at $x \in X$ by endpoint $i \in \mathcal S$ incurs a cost $c(x,i)$, where 
$c: X \times \mathcal N \to \mathbb R \cup \{+\infty\}$ is a given function, taking possibly infinite
values.
Moreover, the product can travel through the network, e.g., the electric transmission system, 
with the variables $p = (p_{ij})_{(i,j) \in \mathcal{A}} \in \mathbb{R}^{|\mathcal{A}|}$ representing
the \emph{flows} along the arcs,
where $|\cdot|$ denotes the cardinality of a set.
Sending a flow $p_{ij}$ along arc $(i,j) \in \mathcal{A}$ incurs a cost $c_{ij}(p_{ij})$, e.g., 
corresponding to active power losses, for given functions $c_{ij}: \mathbb R \to \mathbb R \cup \{+\infty \}$. Illustrative examples are shown on Figures \ref{fig:network_gray} and \ref{fig:partition_elec}
in Section \ref{section:simulations}.

The overall objective is to simultaneously determine how to assign each consumer of $X$ to an endpoint 
in $\mathcal{S}$ and route the product through the network, so as to satisfy demand and minimize
the combined costs of assignment and routing.
A deterministic assignment of consumers in $X$ to endpoints in $\mathcal S$ can be described by a map $T: X \to \mathcal{S}$.
The total demand assigned to endpoint $i\in \mathcal{S}$ is then given by $(T_{\#}\mu)_i=\mu(\{x\in X: T(x)=i\}) = \mu(T^{-1}(i))$, 
where $T_{\#} \mu$ denotes the pushforward measure of $\mu$ by $T$. The set~$T^{-1}(i) \subset X$ defines a \emph{cell} of points assigned to endpoint~$i$, and these cells for $i \in \mathcal S$
form a partition of $X$.
The combined spatial assignment and network flow optimization problem is then formulated as follows
\begin{subequations}
\label{MP}
\begin{align}  
\inf_{\substack{T: X \to \mathcal S \\ p \in \mathbb{R}^{|\mathcal{A}|}}} \; & \int_X c(x, T(x))\,\mathrm{d}\mu(x) + \sum_{(i,j)\in\mathcal{A}} c_{ij}(p_{ij}) \label{eq:obj} \\
\text{s.t.} \quad\,& \sum_{(i,j)\in\mathcal{A}} p_{ij} - \sum_{(j,i)\in\mathcal{A}} p_{ji} = s_i - (T_{\#}\mu)_i, \; \forall\, i\in\mathcal{S}, \label{eq:constraint_station} \\
& \sum_{(i,j)\in\mathcal{A}} p_{ij} - \sum_{(j,i)\in\mathcal{A}} p_{ji} = s_i, \;\;\forall\, i\in\mathcal{N}\setminus\mathcal{S}, \label{eq:constraint_demand} \\
& a_{ij} \leq p_{ij} \leq b_{ij}, \; \forall\, (i,j)\in\mathcal{A}, \label{eq:limits}
\end{align} 
\end{subequations}
where in~\eqref{eq:limits} the scalars $a_{ij}$ and $b_{ij}$, for $(i,j) \in \mathcal A$, 
are lower and upper limits on arc flows,
and~\eqref{eq:constraint_station} and~\eqref{eq:constraint_demand} represent the conservation of flow at each node.
In particular, we distinguish between the flow constraints \eqref{eq:constraint_station} at the endpoints~$i \in \mathcal S$, which must serve the demand $(T_{\#}\mu)_i$, and the flow constraints~\eqref{eq:constraint_demand} 
at the other nodes that only route traffic through the network.  

In this paper, we make the following assumptions.
\begin{assumption}  
\label{regularity_assumption}
We assume that:
\begin{enumerate}[(a)]
\item The assignment cost function $c: X \times \mathcal{S} \to \mathbb{R}\cup\{+\infty\}$ is lower semi-continuous and bounded from below.     \label{assump: assignment cost}
\item The domain $X$ is compact. 
\label{assump: compacity}
\item The arc cost functions $c_{ij}: \mathbb{R}\to \mathbb{R}\cup\{+\infty\}$ are convex and lower semi-continuous for all $(i,j) \in \mathcal{A}$.    \label{assump: cij convex}
\item Problem \eqref{MP} admits a feasible solution.
\label{assump: solutions exist}
\end{enumerate}
\end{assumption}
In Assumption~\ref{regularity_assumption}, \eqref{assump: assignment cost} and \eqref{assump: compacity}
are standard in the OT literature, with \eqref{assump: compacity} being satisfied in practical applications
and useful to simplify technical arguments \cite{Villani:2003book:topicsOT}. Property \eqref{assump: cij convex} is necessary to leverage duality results and also standard for MCF problems. Assumption~\ref{regularity_assumption}-\eqref{assump: solutions exist} is non-trivial
and implies in particular, by summing constraints \eqref{eq:constraint_station} and \eqref{eq:constraint_demand}, 
that we must have $\mu(X)=\sum_{i\in\mathcal{N}} s_i$, i.e., an equilibrium between demand and supply, 
in order for the flow conservation constraints to admit a feasible solution.

The objective of this work is to develop an algorithm to solve the optimization problem~\eqref{MP},
which moreover admits a distributed implementation by the nodes of the network, i.e., each node should 
execute operations that only require information exchanges with their neighbours in the graph $G$. 
In addition, the method of assigning consumers to endpoints should also scale to large-scale problems. 

\section{Characterization of an Optimal Solution}

\subsection{Kantorovich Relaxation and its Dual}
\label{section:relaxation}

From an OT point of view, \eqref{MP} is a type of Monge problem~\cite{Villani:2003book:topicsOT}, 
because the assignment of consumers to endpoints takes the form of a map.
For such problems, it is generally useful to introduce the corresponding Kantorovich relaxation,
by replacing the transport map $T$ with a transport plan $\pi \in \mathcal{M}_+(X\times\mathcal{S})$, 
i.e., a non-negative measure on $X \times \mathcal{S}$. Intuitively, this change allows for randomized assignments of consumers to endpoints.
The relaxed problem reads
\begin{subequations}
\label{KP}
\begin{align}
    \inf_{\substack{\pi \in \mathcal{M}_+(X\times\mathcal{S}) \\ p \in \mathbb{R}^{|\mathcal{A}|}}}  &
    \sum_{i\in\mathcal{S}}\int_X c(x,i)\, \mathrm{d}\pi(x,i) + \sum_{(i,j)\in\mathcal{A}} c_{ij}(p_{ij})     \label{eq:obj_relax} \\
    \text{s.t.} \qquad &\hspace{-0.5cm}\sum_{(i,j)\in\mathcal{A}} p_{ij} - \sum_{(j,i)\in\mathcal{A}} p_{ji} = s_i - 
    \pi(X,i), \; \forall\, i\in\mathcal{S},   \label{eq: Kantorovich assignment constraint} \\
    & \hspace{-0.5cm}\sum_{(i,j)\in\mathcal{A}} p_{ij} - \sum_{(j,i)\in\mathcal{A}} p_{ji} = s_i,\; \forall\, i\in\mathcal{N}\setminus\mathcal{S}, \nonumber \\
    & 
    \hspace{-0.25cm}\sum_{i\in\mathcal{S}} \pi(A,i) = \mu(A),\; \forall\, A\subset \mathcal B(X), \label{eq: marginal constraint} \\
    &\hspace{-0.25cm} a_{ij}\le p_{ij}\le b_{ij},\; \forall\, (i,j)\in\mathcal{A},\label{eq: bound constraint}
\end{align} 
\end{subequations}
where $\mathcal B(X)$ denotes the $\mu$-measurable subsets of $X$.
Constraint~\eqref{eq: marginal constraint} ensures that $\mu$ is the first marginal of $\pi$, 
and in~\eqref{eq: Kantorovich assignment constraint} the quantity $\pi(X,i)$ represents again the 
total demand assigned to $i$. Problem~\eqref{KP} is a relaxation of~\eqref{MP}, because~\eqref{MP} is obtained by restricting $\pi$ to measures induced by deterministic maps, i.e., 
of the form $\pi = (\text{id}_X,T)_{\#} \mu$, where $\text{id}_X$ is the identity map of $X$
and $T: X \to \mathcal S$.

The benefit of this relaxation is that~\eqref{KP} is now a linear program, 
although still infinite-dimensional in general, which satisfies useful duality properties~\cite{Villani:2003book:topicsOT}. In particular, we establish the following key duality result, whose proof is sketched in Appendix~\ref{appdx: duality proof}.

\begin{theorem} 
\label{thm: duality result}
Under Assumption~\ref{regularity_assumption}, the optimal value of \eqref{KP} is equal to 
\begin{equation} \label{eq: dual problem}
\sup_{\psi \in \mathbb R^{|\mathcal N|}} q(\psi),
\end{equation}
where the dual function $q: \mathbb R^{|\mathcal N|} \to \mathbb R$ is given by
\begin{align}
q(\psi)&=\int_X \min_{i\in\mathcal{S}}\{c(x,i)-\psi_i\}\,\mathrm{d}\mu(x)+\sum_{i\in\mathcal{N}} \psi_i\,s_i
\label{eq: dual function}
\\
&\quad+\sum_{(i,j)\in\mathcal{A}} \min_{p\in[a_{ij},b_{ij}]} \{ c_{ij}(p)-(\psi_i-\psi_j)p \}.  \nonumber
\end{align}
Moreover, the infimum in \eqref{KP} is attained at an optimal solution $(\pi^*, p^*) \in \mathcal M_+(X \times \mathcal S) \times \mathbb R^{|\mathcal A|}$.
\end{theorem}

Note that \eqref{eq: dual problem} is a finite-dimensional optimization problem,
in contrast to the primal problem \eqref{KP}, a feature that we exploit to design 
our algorithm in Section \ref{section:algorithm}.

\subsection{Reconstructing the Primal Optimal Solution}
\label{section: primal Monge}

To recover an optimal solution for the original problem~\eqref{MP}, we introduce
a few additional assumptions. The first is commonly assumed to guarantee the tightness of the Kantorovich relaxation 
for SDOT problems \cite{SDOT_Distributed}.
\begin{assumption}  \label{HypSDOT}
For every pair of distinct nodes $i,j\in\mathcal{S}$ and every real number $r\in\mathbb{R}$, 
the set $\{x\in X: c(x,i)-c(x,j)=r\}$ has $\mu$-measure zero.
\end{assumption}

Define, for each $i \in \mathcal S$ and $\psi \in \mathbb R^{|\mathcal N|}$, the generalized 
Laguerre cell \cite{Aurenhammer:SIAM87:powerdiagrams, peyré2019computationaloptimaltransport}
\begin{equation*}    
\mathrm{Lag}_i(\psi) = \left \{ x\in X: c(x,i)-\psi_i\le c(x,j)-\psi_j, \forall\, j\in\mathcal{S} \right\}.
\end{equation*}
These cells are used below to define the optimal assignment map. 
Assumption~\ref{HypSDOT} ensures that the intersection of two distinct Laguerre cells, 
where assignment randomization could be beneficial, has $\mu$-measure $0$ and hence
does not contribute to the overall cost.
Note that the Laguerre cells are polytopes when $X = \mathbb R^d$ and 
$x \mapsto c(x, i)$ is the squared Euclidean distance between $x$ and the position of
node $i \in \mathcal S$ \cite{Aurenhammer:SIAM87:powerdiagrams}.
Next, we introduce the following technical assumption.

\begin{assumption}  \label{assump: dual maximizer exists}
A dual optimal solution $\psi^* \in \mathbb R^{|\mathcal N|}$ 
maximizing \eqref{eq: dual problem} exists.
\end{assumption}
Explicit conditions under which Assumption~\ref{assump: dual maximizer exists} 
is satisfied will be developed in future work.
It is empirically satisfied in our numerical simulations in Section \ref{section:simulations},
and it is known to be satisfied for the standard SDOT problem under Assumption~\ref{regularity_assumption}-\eqref{assump: compacity}, see \cite[Lemma 9]{lemma9}.
Finally, the next assumption, which strengthens Assumption \ref{regularity_assumption}, is often made to simplify the analysis of the MCF component of the problem.  

\begin{assumption}  \label{HypMCF} 
For each arc $(i,j)\in\mathcal{A}$, the cost function $c_{ij}:\mathbb{R}\to\mathbb{R}$ is strictly convex.
\end{assumption}

Assumption \ref{HypMCF} ensures that for all $(i,j) \in \mathcal A$, the minimizer
\begin{equation}    \label{eq: arc price}
p_{ij}(\psi) \coloneqq \arg\min_{p\in[a_{ij},b_{ij}]} \{ c_{ij}(p)-(\psi_i-\psi_j)p \},
\end{equation}
exists and is unique.
The following proposition provides a means to compute an optimal solution for the original problem~\eqref{MP}, assuming an optimal dual solution is known. It also shows that the relaxation~\eqref{KP} is tight.

\begin{proposition}
\label{thm:primal_reconstruction}
Under Assumptions \ref{regularity_assumption}, \ref{HypSDOT}, 
\ref{assump: dual maximizer exists}, and \ref{HypMCF}, 
let $\psi^*$ be a maximizer of \eqref{eq: dual problem}, 
let $p^*_{ij} \coloneqq p_{ij}(\psi^*)$ for each arc $(i,j)\in\mathcal{A}$, 
computed from \eqref{eq: arc price}, and let $p^* \in \mathbb R^{|\mathcal A|}$ 
be the vector with components $p^*_{ij}$. 
Define the assignment map $T^*(x) := \arg\min_{i\in\mathcal{S}} \{ c(x,i)-\psi^*_i \}$, 
for every $x\in X$, with ties between endpoints in the minimization broken arbitrarily,
if any.  
Then, the pair $(T^*,p^*)$ is an optimal solution for~\eqref{MP}.
\end{proposition}

Note that the map $T^*$ in Proposition~\ref{thm:primal_reconstruction} specifies
that all points in $\mathrm{Lag}_i(\psi^*)$ should be assigned to endpoint 
$i \in \mathcal S$. The proof of this proposition is sketched in Appendix \ref{appdx: recovering a solution proof}.

\section{Distributed Algorithm}
\label{section:algorithm}
In this section we propose a distributed algorithm to solve problem \eqref{MP}.
Based on the previous analysis, one can compute an optimal dual solution for \eqref{eq: dual problem} 
and then recover an optimal primal solution using Proposition \ref{thm:primal_reconstruction}.

\subsection{Concavity and Supergradient of the Dual Function}

In this section we show formally that the dual function \eqref{eq: dual function}
is concave, which is expected from duality theory. We also provide an explicit 
expression for its supergradient, which we exploit later 
to design a gradient ascent algorithm.

\begin{proposition} 
\label{prop: gradient}
The dual function $q$ defined in~\eqref{eq: dual function} is concave. 
Moreover, under Assumptions~\ref{regularity_assumption},~\ref{HypSDOT}, and~\ref{HypMCF}, 
a supergradient $g(\psi)$ at $\psi \in \mathbb R^{|\mathcal N|}$ has components, 
for each $i\in\mathcal{S}$,
\begin{equation}    \label{eq: supergradient component for S}
g(\psi)_i = s_i - 
\mu(\mathrm{Lag}_i(\psi))
- \sum_{(i,j)\in\mathcal{A}}p_{ij}(\psi) + \sum_{(j,i)\in\mathcal{A}}p_{ji}(\psi),
\end{equation}
and, for each $i\in\mathcal{N}\setminus\mathcal{S}$,
\begin{equation}    \label{eq: supergradient component not S}
g(\psi)_i = s_i - \sum_{(i,j)\in\mathcal{A}}p_{ij}(\psi) + \sum_{(j,i)\in\mathcal{A}}p_{ji}(\psi).
\end{equation}
\end{proposition}

\begin{proof} We decompose the dual function $q$ in \eqref{eq: dual function} as
$q(\psi)=q_\text{SDOT}(\psi)+q_\text{MCF}(\psi)$, with
\begin{align*}
q_\text{SDOT}(\psi) &= \int_X \min_{i\in\mathcal{S}}\{c(x,i)-\psi_i\}\,\mathrm{d}\mu(x),
\end{align*}
and, 
\[
q_\text{MCF}(\psi)=\sum_{i\in\mathcal{N}} \psi_i s_i+\sum_{(i,j)\in\mathcal{A}} \min_{p\in[a_{ij},b_{ij}]} 
\{ c_{ij}(p)-(\psi_i-\psi_j)p \},
\]
by analogy with the dual functions for the SDOT problem~\cite{SDOT_Distributed} 
(without the linear term $\sum_{i\in \mathcal{S}} \psi_i s_i$) and the MCF problem~\cite{bertsekas1998network}.
Then, $q_\text{MCF}$ is concave as a sum of linear terms and the minimum of affine functions. 
Under Assumption~\ref{HypMCF}, its supergradient at $\psi$ is given by the 
expression~\eqref{eq: supergradient component not S} for each $i\in\mathcal{N}$,
see \cite[Chapter 5.5.]{BertsekasDistributed}.
Moreover, under Assumption~\ref{HypSDOT}, we can write
\[
q_\text{SDOT}(\psi) = \sum_{i\in\mathcal{S}}\int_{\mathrm{Lag}_i(\psi)} (c(x,i)-\psi_i)\, \mathrm{d}\mu(x).
\]
As shown in \cite[Theorem 4]{SDOT_Distributed}, this function
is concave and its supergradient at $\psi$ has as component $- \mu(\mathrm{Lag}_i(\psi))$
for each $i\in\mathcal{S}$.
By addition, $q$ is concave and its supergradient is given by~\eqref{eq: supergradient component for S} for each
$i\in\mathcal{S}$ and ~\eqref{eq: supergradient component not S} for each
$i\in\mathcal{N}\setminus S$.
\end{proof}

\subsection{Supergradient Algorithm and Distributed Implementation}

Starting from an arbitrary initial value $\psi^0 \in \mathbb R^{|\mathcal N|}$, 
an iterative supergradient ascent algorithm to maximize
\eqref{eq: dual problem} takes the form
\begin{equation}    \label{eq: supergradient algo}
\psi^{k+1} = \psi^k + \gamma_k \, g(\psi^k),
\end{equation}
with $g(\psi)$ given in Proposition \ref{prop: gradient} and the positive scalar stepsizes $\{\gamma_k\}_{k \geq 0}$ 
satisfying the standard conditions
$$
\sum_{k=0}^\infty \gamma_k = +\infty \quad \text{and} \quad \sum_{k=0}^\infty \gamma_k^2 < +\infty.
$$
Under these conditions, and noting that the gradients in Proposition \ref{prop: gradient} are
uniformly bounded,
the sequence $q(\psi^k)$ converges to an optimal value $q^*$ of \eqref{eq: dual problem} as $k \to +\infty$. Moreover, under Assumption \ref{assump: dual maximizer exists}, the iterates $\psi_k$
also converge to an optimal solution $\psi^*$ \cite[Proposition 8.2.6]{Bertsekas:book03:convex},
which can then be used to compute the assignments and flows using Proposition~\ref{thm:primal_reconstruction}.
We can now describe a distributed
method to solve problem \eqref{MP}, outlined in Algorithm~\ref{alg:distributed}.

\begin{algorithm}[htpb!]
\caption{Agent $i$'s flow and assignment computations}  \label{alg:distributed}
\begin{algorithmic}[1]
\STATE \textbf{Input for agent/node $i \in \mathcal N$:} 
$c(x,i)$, $s_i$, flow bounds $a_{ij},b_{ij}$ for $j$ s.t. $(i,j) \in \mathcal A$,
threshold $\varepsilon > 0$
\STATE \textbf{Output:} Optimal outgoing flows $\{p_{ij}^*\}_{j:(i,j) \in \mathcal A}$; 
and assignment $(T^*)^{-1}(i)$ if $i \in \mathcal S$
\STATE \textbf{Initialization:} Initialize dual variable $\psi_i$. $k \gets 0$ 
\REPEAT
    \STATE \label{algo: local flows computation - start}
    Collect $\psi_j$ from neighbours $j$ s.t. $(i,j) \in \mathcal A$
    \STATE \label{algo: local flows computation}
    Compute local flows, for all $j$ such that $(i,j)\in\mathcal{A}$:
    $$
    p_{ij} \gets \arg\min_{p\in[a_{ij},b_{ij}]}\{ c_{ij}(p)-(\psi_i-\psi_j)p\} 
    $$
    \STATE \label{algo: local flows computation - end}
    Collect values $p_{ji}$ from $j$ s.t. $(j,i)\in\mathcal{A}$    
    \IF{$ i\in\mathcal{S} $}    
        \STATE Cell mass computation : $m_i \gets \mu(\mathrm{Lag}_i(\psi))$ \label{algo: cell mass computation 1}
        \STATE \label{algo: gradient comput 1} $g_i \gets s_i - m_i - \sum_{(i,j)\in\mathcal{A}} p_{ij} + \sum_{(j,i)\in\mathcal{A}} p_{ji}$
    \ELSE 
        \STATE \label{algo: gradient comput 2} $ g_i \gets s_i - \sum_{(i,j)\in\mathcal{A}} p_{ij} + \sum_{(j,i)\in\mathcal{A}} p_{ji}$
    \ENDIF
    \STATE Dual update: $\psi_{i,old} \gets \psi_i$, $\psi_i \gets \psi_i+\gamma_k g_i$, $k \gets k+1$
\UNTIL{$|\psi_i - \psi_{i,old}| < \varepsilon$} \label{algo: termination}
\STATE \textbf{Primal reconstruction:}
$p_{ij}^*\gets p_{ij}$; and if $i \in \mathcal S$,
$(T^*)^{-1}(i) = \mathrm{Lag}_i(\psi)$
\end{algorithmic}
\end{algorithm}

Assume that a computing agent at each node $i \in \mathcal{N}$ within the network
updates the dual variable $\psi_i$ according to \eqref{eq: supergradient algo}, for which
it needs to compute the supergradient component $g_i(\psi)$ given in Proposition \ref{prop: gradient}.
At each iteration 
the agents compute the current flows $\{p_{ij}\}_{(i,j) \in \mathcal A}$
on Lines \ref{algo: local flows computation - start}-\ref{algo: local flows computation - end}
by exchanging information only with their outgoing and incoming neighbours. 
For endpoint nodes $i \in \mathcal{S}$, an additional step 
on Line \ref{algo: cell mass computation 1} is required to determine 
the measure $m_i = \mu(\mathrm{Lag}_i(\psi))$ of its current Laguerre cell 
in order to compute the supergradient component \eqref{eq: supergradient component for S}.
This step depends a priori on the variables $\psi_j$ of possibly all the other endpoints $j \in \mathcal S$.
Most practical instances feature a finite set of customers and 
a discrete measure $\mu$ encoding their individual demands. 
At iteration $k$, each endpoint $i\in\mathcal S$ needs to compare the adjusted 
costs $c(x,i)-\psi_i^{k}$ with those of the other endpoints.  
A customer at $x$ is matched to the endpoint that offers the lowest value. 
Then, we can compute $m_i=\sum_{x\in\mathrm{Lag}_i(\psi)}\mu(x)$ 
the total demand of the customers lying in the Laguerre cell of~$i$.
For some cost functions $c(x,i)$, e.g., the squared distance between $x$ 
and the position of node $i$, the geometric properties of the Laguerre cells can 
be exploited so that an endpoint only needs to compare the adjusted costs 
with other endpoints that share a cell boundary with it, 
see, e.g., \cite{Kwok:IJRNC10:powerDiags}.

\begin{remark}
The discretization of the space $X$ to compute the 
masses $m_i$ deterministically can be replaced by a stochastic integration method, generating samples according
to $\mu$ for which the endpoints then compare their adjusted costs, as discussed in~\cite{SDOT_Distributed}
for SDOT. The iterates~\eqref{eq: supergradient algo} then become a stochastic gradient algorithm, 
again with convergence guarantees.
\end{remark}

Finally, the agents need to detect convergence in a distributed manner. 
A possible test used on Line \ref{algo: termination} 
is that each node stops updating its dual variable $\psi_i$ when its variation 
stays below a certain threshold $\varepsilon$.
Each agent may also wait until the variables of its neighbours have 
converged as well before recomputing the final flows (as on Line \ref{algo: local flows computation})
and Laguerre cells.
Regarding the assignment computation, each endpoint can determine the points 
$x \in X$ belonging to its Laguerre cell. Alternatively in some applications, 
the final weights $\psi^k$ can be broadcast to the customers, who can then 
determine their assigned endpoint automatically.

\section{Numerical Simulations}
\label{section:simulations}

In this section, we illustrate our method through two numerical examples. 
First, we test Algorithm~\ref{alg:distributed}
on a synthetic example. 
Second, we briefly explore the applicability of the method to a more complex scenario relevant to electric power distribution networks.

\subsection{Synthetic Example} \label{section:synthetic}

We first consider a  
simple scenario shown on Figure~\ref{fig:network_gray}, defined on 
a square domain $X \subset \mathbb R^2$ of side length $L=100$. 
A continuous consumer demand distribution is discretized on a $200 \times 200 $~grid, for a total of 40,000 points. 
Each consumer is assigned a mass from a truncated Gaussian density with mean 
$(50,75)$ and standard deviation $25$. 
The backbone network consists of two nodes $\{ S_1, S_2\}$ with a supply $s_i=0.5$, called source nodes, 
and four nodes $\{ I_1, I_2, I_3, I_4 \}$ with zero supply, called interconnection nodes. 
Node positions and arcs are also shown on Figure~\ref{fig:network_gray}. 

The assignment cost $c(x,i)$ between consumer demand at $x$ and node $i$ is 
given by the Euclidean distance between $x$ and the position of node $i$. 
The arc costs are assumed to be quadratic functions of the flows, i.e., 
$c_{ij}(p) \coloneqq d_{ij} \, p^2$, where $d_{ij}$ is the Euclidean distance 
between nodes $i$ and $j$. 
We assume $a_{ij} = -1$ and $b_{ij} = 1$ for all $(i,j) \in \mathcal A$,
and obtain the closed-form local flow update 
at Line \ref{algo: local flows computation} of Algorithm \ref{alg:distributed} given by
$p_{ij} = \operatorname{proj}_{[-1,1]}\left(\frac{\psi_i - \psi_j}{2\,d_{ij}}\right)$.
The algorithm is run 
for $300$ iterations with diminishing step-size $\gamma_k=1/(1+0.01k)$.
Figure~\ref{fig:convergence} shows the convergence of the components of the supergradient $g$ 
toward zero and of the dual variables $\psi$.  
Figure~\ref{fig:partition} illustrates the resulting routing and partitioning solutions. 
Each cell is coloured according to its assigned endpoint node, 
and the network is represented with annotated flows on its arcs. 
Source nodes are labelled with their fixed supply $0.5$, and endpoints display their assigned demand


\begin{figure}[htbp]
    \centering
    \begin{subfigure}[b]{0.48\columnwidth}
        \centering
        \includegraphics[width=\linewidth]{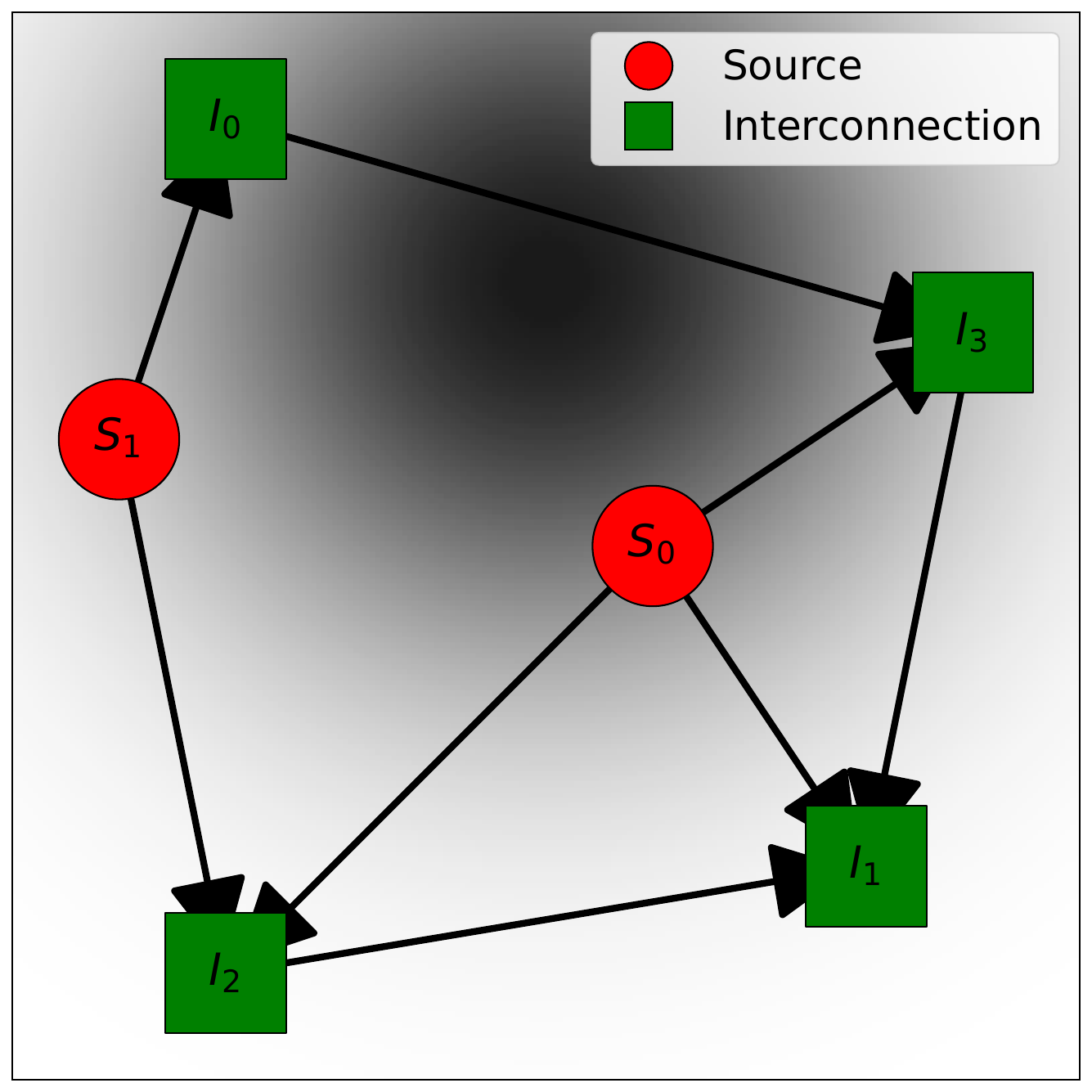}
        \caption{}
        \label{fig:network_gray}
    \end{subfigure}
    \hfill
    \begin{subfigure}[b]{0.48\columnwidth}
        \centering
        \includegraphics[width=\linewidth]{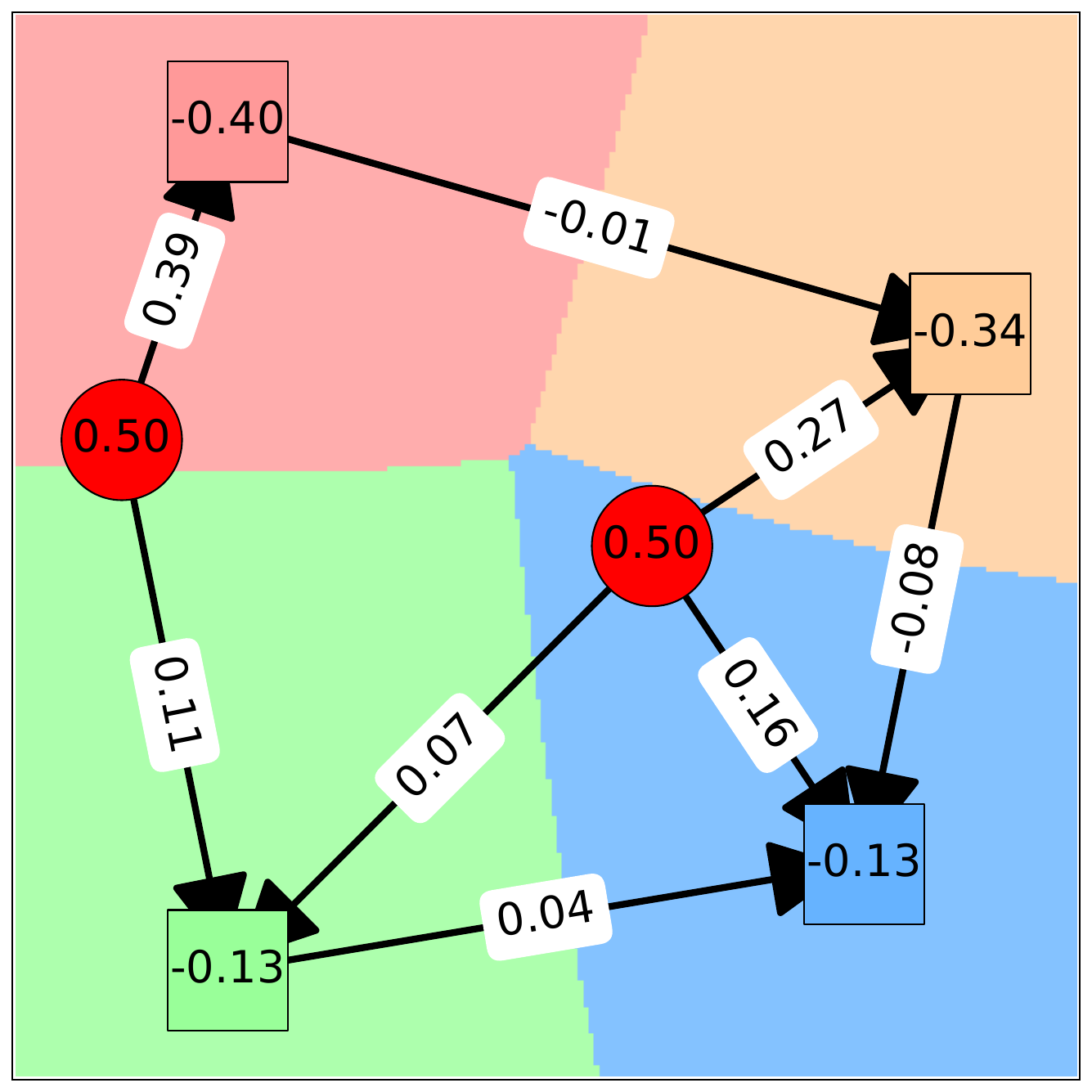}
        \caption{}
        \label{fig:partition}
    \end{subfigure}
    \caption{Synthetic example: (a) simple network and demand density in grayscale, (b) partitioning and routing solution.}
    \label{fig:network_partition}
\end{figure}


\begin{figure}[htbp]
    \centering
    \begin{subfigure}[b]{0.48\columnwidth}
        \centering
        \includegraphics[width=\linewidth]{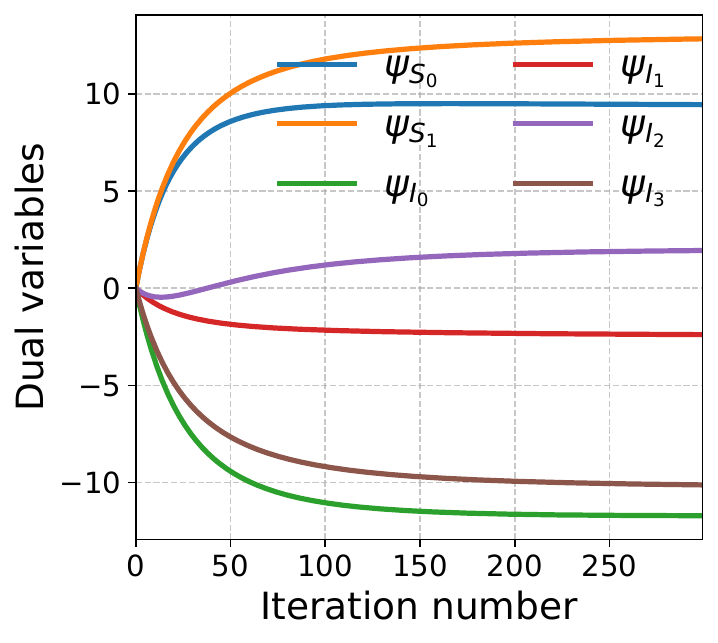}
        \caption{}
        \label{fig:convergence_psi}
    \end{subfigure}
    \hfill
    \begin{subfigure}[b]{0.48\columnwidth}
        \centering
        \includegraphics[width=\linewidth]{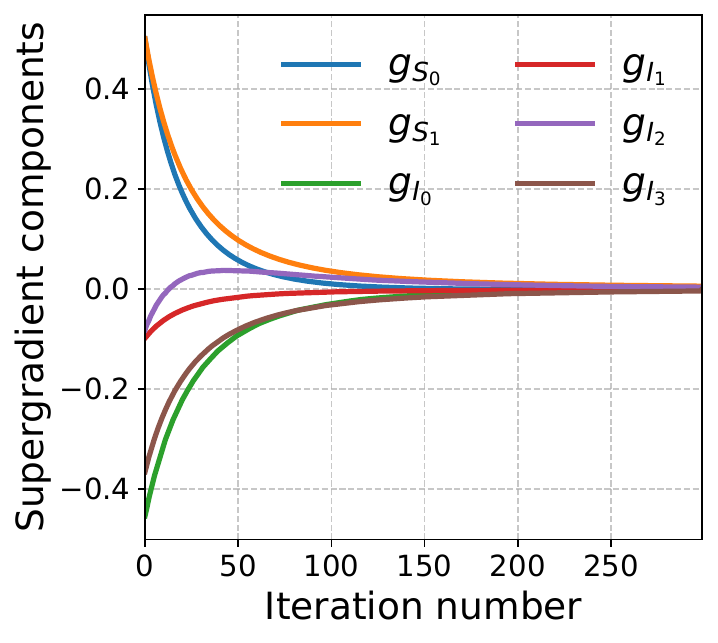}
        \caption{}
        \label{fig:convergence_grad}
    \end{subfigure}
    \caption{Convergence of (a) dual variables $\psi$ and (b) supergradient $g$.}
    \label{fig:convergence}
\end{figure}

\subsection{Electric Power Network Example}
\label{section: electrical network ex}

We now discuss an application to an electrical power network consisting of a transmission 
system linking generators to substations and a distribution network connecting substations 
to consumers. Consumer demand, initially discrete, is approximated by a continuous distribution. 
Substations 
are endpoint nodes $\mathcal{S}$ in our formulation.

An important problem in electrical networks is the radial reconfiguration of the distribution network: 
activating or deactivating switches on lines to ensure that each consumer is served via exactly one 
distribution path to a unique substation, forming tree-like (radial) subnetworks rooted at substations, see Figure~\ref{fig:partition_elec}. 
Solving large-scale reconfiguration problems exactly is computationally challenging~\cite{ReconfigurationSurvey}, so that the coupling with 
the transmission network optimization problem is usually neglected. 
Here we use Algorithm \ref{alg:distributed} to develop a heuristic producing
an initial assignment of customers to substations while minimizing transmission costs. 
This enables subsequent parallelized radial reconfiguration per substation, 
greatly reducing the computational burden.

We assume that the arc costs on the transmission network are $c_{ij}(p) = r_{ij} \, p^2$, 
representing power losses, where $r_{ij}$ is the line resistance and $p$ is the power flow. 
We take $r_{ij}$ as the distance $d_{ij}$ in this example. 
Next, to approximately capture the power losses in the distribution network
through the simple assignment cost $\int c(x,T(x))\,\mathrm d\mu(x)$, we choose 
for $c(x,i)$ the geodesic (shortest‑path) resistance through this network
between a customer $x$ and the substation $i$.
The quality of this heuristic choice will be explored in future work.
This choice also ensures an essential connectivity property for the resulting Laguerre cells: 
all consumers assigned to a substation can be physically connected to it via paths lying entirely 
within their assigned cell. Briefly, if a consumer $x$ belongs to the Laguerre cell $\mathrm{Lag}_i(\psi)$ 
associated with substation $i$ and vector $\psi \in \mathbb{R}^{|\mathcal{S}|}$, all intermediate 
nodes $y$ along the shortest path from $x$ to $i$ also belongs to $\mathrm{Lag}_i(\psi)$. This follows directly from the definitions of Laguerre cells and shortest paths, because for all $j\in \mathcal{S}$:
\begin{align*}
c(x,y)+c(y,i)-\psi_i &= c(x,i)-\psi_i \\
& \le c(x,j)-\psi_j \\
& \le c(x,y)+c(y,j)-\psi_j,
\end{align*}
which implies $c(y,i)-\psi_i \le c(y,j)-\psi_j$, hence $y \in \mathrm{Lag}_i(\psi)$.

Figure~\ref{fig:partition_elec} illustrates the partitioning and routing solutions obtained for this power network example. The transmission network topology mirrors the synthetic example described in Section~\ref{section:synthetic}. The distribution network was generated by randomly placing 1000 consumer nodes within the domain, each assigned a uniformly random demand value. To ensure a realistic network structure, each consumer node was connected to its two or three geographically closest neighbours. Figure~\ref{fig:partition_elec} displays the resulting Laguerre cells, with consumer nodes coloured according to their assigned substation, alongside optimized flows within the transmission network. 
The algorithm parameters used match those of the synthetic example.

\begin{figure}[htbp]
    \centering
    \includegraphics[width=1\columnwidth]{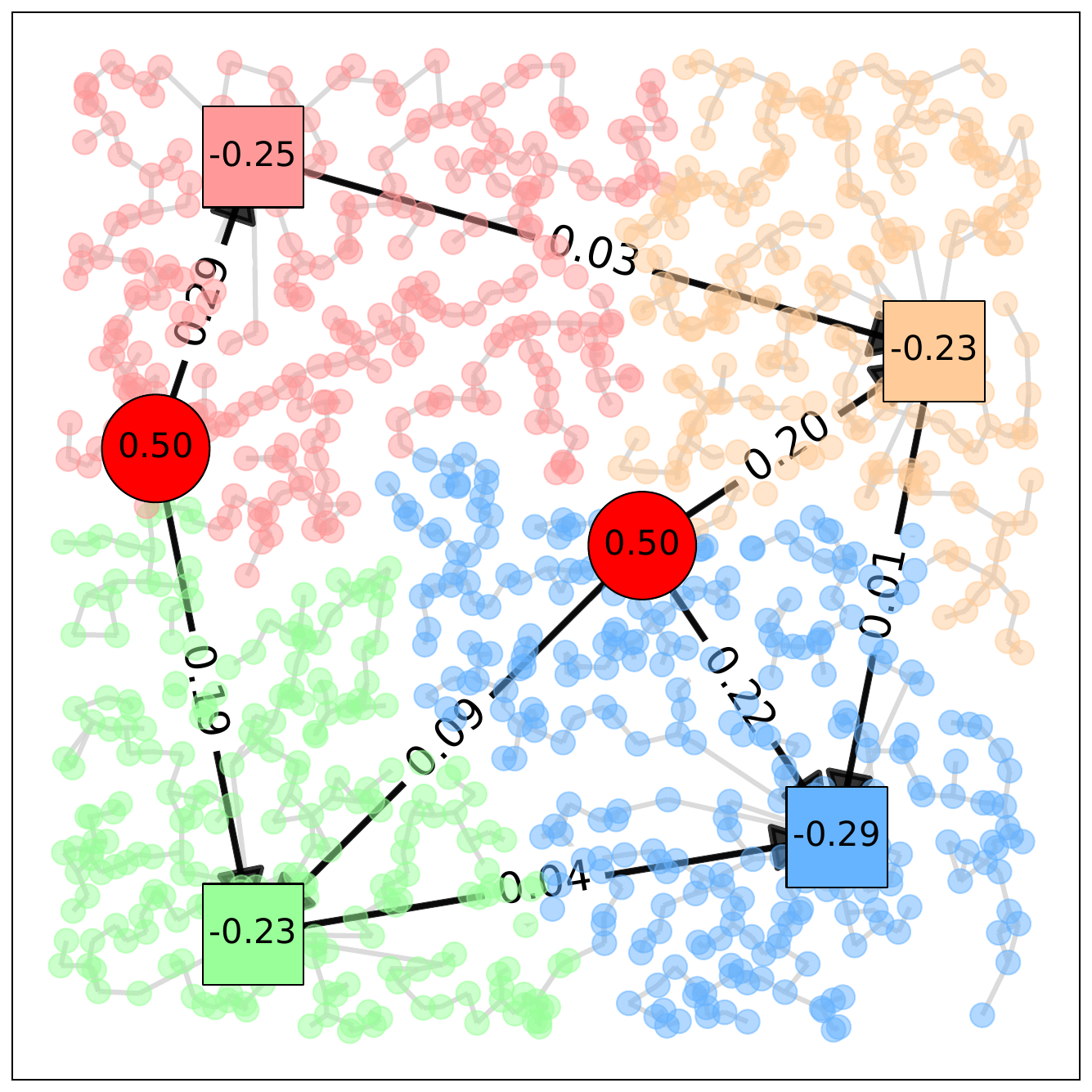}
    \caption{Optimal spatial partitioning and transmission network routing for the electrical network example with 1000 consumers. Each endpoint is the root of a tree connecting its assigned
    customers.}
    \label{fig:partition_elec}
\end{figure}

\section{Conclusion}\label{section:conclusion}
We have studied a combined space partitioning and network flow optimization problem motivated by the operation of large-scale networks serving spatially distributed consumers. By adopting an optimal transport perspective, we derived a dual formulation and designed a distributed algorithm relying only on local information exchanges. Numerical experiments confirmed the effectiveness of the method and illustrated its relevance to power distribution network reconfiguration. Future work will investigate extensions of the framework, in particular incorporating flexible generation through additional decision variables at network nodes.

\appendix
\subsection{Proof Sketch of Theorem \ref{thm: duality result}}
\label{appdx: duality proof}

Theorem \ref{thm: duality result} can be proved by following a general methodology 
for OT problems, as in \cite[Theorem 1.3]{Villani:2003book:topicsOT}.
Recall that for any function $f: E\to \mathbb{R}\cup\{+\infty\}$ on a normed vector space $E$ with dual $E^*$, 
its conjugate $f^*:E^*\to \mathbb{R}\cup\{+\infty\}$ is defined by $f^*(y) = \sup_{x\in E} \left\{ \langle x, y \rangle - f(x) \right\}.$

The methodology relies on the following version of the Fenchel–Rockafellar duality theorem
\cite[Theorem 1.9]{Villani:2003book:topicsOT}. 

\begin{theorem}[Fenchel–Rockafellar] 
\label{thm: Fenchel-Rockafellar}
Let $E$ be a normed vector space and $E^*$ its dual. If $\Theta,\Xi : E\to \mathbb{R}\cup\{+\infty\}$ 
are convex functions and if there exists $z_0\in E$ such that $\Theta(z_0)$ and $\Xi(z_0)$ are finite 
and $\Theta$ is continuous at $z_0$, then
\[
\inf_{z\in E}\Bigl\{\Theta(z)+\Xi(z)\Bigr\} = -\max_{z^*\in E^*}\Bigl\{-\Theta^*(-z^*)-\Xi^*(z^*)\Bigr\}.
\]
\end{theorem}

To prove Theorem \ref{thm: duality result}, we define suitable convex functions~$\Theta$ and~$\Xi$ 
on a suitable space $E$ and apply Theorem \ref{thm: Fenchel-Rockafellar}.
In our case, 
we take $E = \mathcal{C}_b(X \times \mathcal{S}) \times \mathbb{R}^{|\mathcal{A}|}$,
where~$\mathcal{C}_b(X \times \mathcal{S})$ denotes the space of continuous bounded functions on $X \times \mathcal{S}$ 
equipped with the supremum norm.
By Riesz's theorem, its topological dual space can be identified with the space of finite signed Radon measures on $X\times \mathcal{S}$, denoted by $\mathcal{M}(X \times \mathcal{S})$. Thus, the dual space is $E^* = \mathcal{M}(X \times \mathcal{S}) \times \mathbb{R}^{|\mathcal{A}|}$,
with the duality pairing between $E$ and $E^*$ given by 
\[
\langle (u,v), (\pi,p) \rangle = \int_{X \times \mathcal{S}} u(x,i)\,\mathrm{d}\pi(x,i) + \sum_{(i,j)\in \mathcal{A}} v_{ij}\,p_{ij}.
\]
We then define $\Theta(u,v) = \Theta_\text{OT}(u) + \Theta_\text{MCF}(v)$, where
\[
\Theta_\text{OT}(u)=
\begin{cases}
0, & \text{if } u(x,i)\ge -c(x,i) \; \forall\,(x,i)\in X\times \mathcal{S},\\
+\infty, & \text{otherwise,}
\end{cases}
\]
as in \cite{Villani:2003book:topicsOT}, and $\Theta_\text{MCF}(v)=\sum_{(i,j)\in \mathcal{A}}\tilde{c}_{ij}^*(-v_{ij})$, with $\tilde{c}_{ij}(p)=c_{ij}(p)+\delta_{[a_{ij},b_{ij}]}(p).$

As for the function $\Xi: E \to \mathbb{R}\cup\{+\infty\}$, extending the idea in \cite[Theorem 1.3]{Villani:2003book:topicsOT}, 
if there exist functions $\varphi\in C_b(X)$ and $\psi\in\mathbb{R}^{|\mathcal{N}|}$ such that 
$u(x,i)=\varphi(x)+\psi_i$ for all $(x,i)\in X\times\mathcal{S}$
$v_{ij}=\psi_j-\psi_i$ for all $(i,j)\in\mathcal{A}$, then
\[
\Xi(u,v)=\int_X \varphi(x)\,\mathrm{d}\mu(x) + \sum_{i\in\mathcal{N}} \psi_i\,g_i,
\]
otherwise, we set $\Xi(u,v)=+\infty$. The function $\Xi$ is well-defined \cite[p. 27]{Villani:2003book:topicsOT}.
With some sign changes in the variable definitions, we have that
\[
\inf_{(u,v)\in E}\{\Theta(u,v)+\Xi(u,v)\} = -\sup_{\psi\in\mathbb{R}^{|N|}}q(\psi).
\]

Next, 
after some calculations, leveraging $\tilde{c}_{ij}^{**}=\tilde{c}_{ij}$ under 
Assumption \ref{regularity_assumption}-\eqref{assump: cij convex} (convexity), 
we can show that
\[
\Theta^*\bigl(-(\pi,p)\bigr)=\int_{X\times \mathcal{S}}c(x,i)\,\mathrm{d}\pi(x,i)+\sum_{(i,j)\in \mathcal{A}}c_{ij}(p_{ij}),
\]
provided that $\pi\in \mathcal{M}(X\times \mathcal{S})$, $\pi\ge0$ and $p_{ij}\in[a_{ij},b_{ij}]$, i.e., matching \eqref{eq:obj_relax} under \eqref{eq: bound constraint}, and $+\infty$ otherwise.
For $\Xi^*$, by reparameterizing any $(u,v)$ with finite $\Xi(u,v)$ in terms of $\varphi$ and $\psi$, one obtains that $\Xi^*(\pi,p)$ enforces the marginal condition $\pi(A\times\mathcal{S})=\mu(A)$ for all $A\subset \mathcal B(X)$ (see \eqref{eq: marginal constraint}) and the flow balance constraints (see \eqref{eq: Kantorovich assignment constraint}); that is, $\Xi^*(\pi,p)=0$ if these conditions hold and $+\infty$ otherwise.

It follows that $\sup_{(\pi,p)\in E^*}\{-\Theta^*(-(\pi,p))-\Xi^*(\pi,p)\}$ is exactly the negative 
of the primal problem \eqref{KP}. Thus, by applying  
Theorem \ref{thm: Fenchel-Rockafellar}
we obtain the desired duality result. Assumption~\ref{regularity_assumption}, with the key hypotheses 
that $c$ and $c_{ij}$ are lower semicontinuous and convex, $X$ is compact, and the feasibility 
condition $\mu(X)=\sum_{i\in N}g_i$, ensures that $\Theta$ and $\Xi$ satisfy the convexity, continuity
and qualification conditions required to apply the theorem.

\subsection{Proof Sketch of Proposition \ref{thm:primal_reconstruction}}
\label{appdx: recovering a solution proof}
     
Let $(\pi^*, p^*)$ be an optimal solution for (2) and $\psi^*$ be an optimal dual solution 
for \eqref{eq: dual problem}. 
Since $(\pi^*, p^*)$ is feasible, the flow constraints hold, so by multiplying each constraint 
in the primal problem by $\psi^*_i$ and  
using strong duality, we have
\begin{align*}
     q(\psi^*) =& \int_{X\times\mathcal{S}} c(x,i)\,\mathrm{d}\pi^*(x,i) + \sum_{(i,j)\in\mathcal{A}} c_{ij}(p^*_{ij}) \\ 
     &\hspace{-0.5cm}+\sum_{i\in\mathcal{S}} \psi^*_i \left[ -\sum_{(i,j)\in\mathcal{A}} p^*_{ij} + \sum_{(j,i)\in\mathcal{A}} p^*_{ji} + s_i - \pi^*(X,i) \right] \\
     &+\sum_{i\in\mathcal{N\setminus \mathcal{S}}} \psi^*_i \left[ -\sum_{(i,j)\in\mathcal{A}} p^*_{ij} + \sum_{(j,i)\in\mathcal{A}} p^*_{ji} + s_i \right]. 
\end{align*}   
Substituting $q(\psi^*)$ by its definition \eqref{eq: dual function} and bringing all the terms 
to the right-hand side, the $\psi_i^*s_i$ terms vanish and we obtain $\sum_{(i,j)\in\mathcal{A}} B_{ij} + \int_{X\times\mathcal{S}} A(x,i)\,\mathrm{d}\pi^*(x,i)=0$, with 
\begin{align*}
B_{ij} \coloneqq& \, c_{ij}(p^*_{ij})-(\psi^*_i-\psi^*_j)p^*_{ij}\\ 
& \quad-\min_{p\in[a_{ij},b_{ij}]}\{c_{ij}(p)-(\psi^*_i-\psi^*_j)p\},
\end{align*}
and, $A(x,i) = [c(x,i)-\psi^*_i] - \min_{k\in\mathcal{S}}\{c(x,k)-\psi^*_k\},$
using the marginal constraint \eqref{eq: marginal constraint}.
Now, because $A(x,i)\ge0$ and $B_{ij}\ge0$, each term must be zero.

For the flow terms, $B_{ij}(p^*_{ij})=0$ for every arc $(i,j)\in\mathcal{A}$ implies  
$p^*_{ij}=p_{ij}(\psi^*)$, i.e., the minimizer in \eqref{eq: arc price} for $\psi^*$.
For the integral term, we must have $\pi^*$-almost everywhere that $A(x,i)=0$, i.e.,  
$\pi^*$ is concentrated on the set 
\[
\left \{ (x,i) \in X \times \mathcal S: c(x,i) - \psi_i^* = \min_{k\in\mathcal{S}}\{c(x,k)-\psi^*_k\} \right \}.
\]
Hence, for $\mu$-almost all $x$, when the minimizing index $i$ is unique $\pi^*$ must
assign $i$ to $x$, which corresponds to the map $T^*$ in Proposition \ref{thm:primal_reconstruction}. 
When the minimizer is not unique, $\pi^*$ may randomize between them, however under 
Assumption \ref{HypSDOT}, choosing the 
deterministic assignment $T^*$ also for such $x$ has no impact on the cost.
Overall, this shows that $(\operatorname{id}_X,T^*)_{\#}\mu$ and $p^*$ is optimal
for \eqref{KP}, so $(T^*,p^*)$ is optimal for \eqref{MP} and the relaxation 
is tight.

\bibliographystyle{IEEEtran}
\bibliography{references}

\end{document}